\documentclass[a4paper,fleqn]{cas-sc}
\usepackage{amsthm}
\usepackage{amsmath}
\usepackage{xcolor}
\usepackage{graphicx}
\usepackage{dsfont}
\usepackage{booktabs}
\usepackage{multirow}
\usepackage{amsfonts}
\usepackage[authoryear,longnamesfirst]{natbib}

\newtheorem{theorem}{Theorem}[section]
\newtheorem{proposition}[theorem]{Proposition}

\usepackage[authoryear,longnamesfirst]{natbib}

\def\tsc#1{\csdef{#1}{\textsc{\lowercase{#1}}\xspace}}
\tsc{WGM}
\tsc{QE}

\begin{document}
\let\WriteBookmarks\relax
\def\floatpagepagefraction{1}
\def\textpagefraction{.001}
\shorttitle{On the Stochastic Gradient Methods Using No-Replacement Sampling}    

\shortauthors{Boresta et al.}  

\title [mode = title]{On the Batch Size Selection in Stochastic Gradient Methods Using No-Replacement Sampling}  



\author[1]{Marco Boresta}

\cormark[1]

\ead{marco.boresta@iasi.cnr.it}



\author[2]{Alberto {\color{black}De Santis}}

\ead{desantis@diag.uniroma1.it}


\author[2]{Stefano Lucidi}

\ead{lucidi@diag.uniroma1.it}


\affiliation[1]{organization={Institute of Systems Analysis and Computer Science "Antonio Ruberti" (IASI), National Research Council (CNR)},
                addressline={Via dei Taurini, 19}, 
                city={Rome},
                postcode={00185}, 
                state={},
                country={Italy}}

\affiliation[2]{organization={Department of Computer, Control and Management Engineering "Antonio Ruberti", Sapienza University of Rome},
                addressline={Via Ariosto, 25}, 
                city={Rome},
                postcode={00185}, 
                state={},
                country={Italy}}

\cortext[1]{Corresponding author}




\begin{abstract}
Recent stochastic gradient methods that have 
appeared in the literature base their efficiency and global convergence properties on a suitable control of the variance of the gradient batch estimate. This control is typically achieved by dynamically increasing the batch size during the iterations of the algorithm. However, in the existing methods the statistical analysis often relies on sampling with replacement. This particular batch selection appears unrealistic in practice. 

In this paper, we consider a more realistic approach to batch size selection based on sampling without replacement. The consequent statistical analysis is compared to that of sampling with replacement. The new batch size selection method, while still ensuring global convergence, provides a more accurate representation of the variance reduction observed in practice, leading to a smoother and more efficient batch size update scheme. 
\end{abstract}



\maketitle


\section{Introduction}

Batch size selection in stochastic gradient methods is crucial for controlling the variance of gradient estimates, which directly impacts the convergence of optimization algorithms. This is particularly important in large-scale machine learning tasks, where evaluating the full gradient is computationally expensive. Recent works, such as \cite{franchini2023line}, and its extensions \cite{franchini2023learning,franchini2024stochastic}, have introduced efficient strategies to dynamically adjust the batch size. These strategies are grounded in theoretical analyses (see e.g. Theorem 2.2 of \cite{franchini2023learning}) which show that global convergence of the stochastic gradient algorithms can be achieved by controlling the variance of the gradient batch estimate.

In these methods, variance reduction is typically achieved by increasing the batch size during iterations of the algorithm. The variance bounds in \cite{franchini2023learning} and \cite{franchini2024stochastic} refer to the form of the variance of the gradient batch estimate obtained by a statistical analysis that follows the hypothesis of performing a sampling scheme with replacement. This scheme allows multiple selections of the same items within a batch, and this can be computationally inefficient and unrealistic. In practice, multiple choice should be avoided and this can be obtained by sampling without replacement. The selection of each item is no longer independent of the others, leading to a different formula for the variance of the gradient estimate. While both schemes imply that the variance decreases as the batch size increases, the variance formula for the replacement-based approach implies that the batch size increases indefinitely, since the variance must go to zero to ensure global convergence. In contrast, the no-replacement scheme introduces a natural upper bound based on the dataset size and results in smoother batch size updates. The differences between these two approaches become more apparent at larger batch sizes. Other features of stochastic gradient methods with no replacement are studied in \cite{beneventano2024trajectoriessgdreplacement}, \cite{pmlr-v97-gazagnadou19a}, \cite{NIPS2016_c74d97b0}.

\par\smallskip\noindent
Our contribution is to propose a more realistic sampling scheme leading to a batch size selection procedure that can reduce computational overhead while maintaining theoretical convergence properties.

\par\smallskip\noindent
The paper is organized as follows. In Section 2, we analyze batch selection strategies, comparing sampling with and without replacement. Section 3 presents the proposed method and theoretical results. Section 4 describes the numerical experiments performed to validate our approach. Finally, Section 5 concludes the paper and suggests directions for future research.

\section{Batch selection analysis}

In this section, we propose a novel approach for selecting the batch size in stochastic gradient methods. Our scheme builds upon the convergence results  studied in \cite{franchini2023line,franchini2023learning}. Specifically, we aim to address limitations related to controlling the variance of stochastic gradient estimates during the optimization process.

This paper addresses the optimization problem:
\begin{equation}\label{Problem}
    \min_{x \in \mathbb{R}^d} F(x) = \frac{1}{N} \sum_{i=1}^{N} f_i(x),
\end{equation}
where each function $f_i: \mathbb{R}^d \rightarrow \mathbb{R}$ is differentiable. We focus on scenarios where the number of components $N$ is significantly large. In such cases, computing the full function $F(x)$ and the full gradient $\nabla F(x)  = \frac{1}{N}\sum_{i=1}^{N} \nabla f_i(x)$ is impractical. Therefore, we approximate $F(x)$ and $\nabla F(x)$ by using a smaller subset (batch) $S$, with size $N_S<N$, of elementary functions, as follows\footnote{For sake of simplicity $S$ denotes both the batch and the index values that label the items within the batch.}:
$$
F_{N_S}(x)=\frac{1}{N_S}\sum_{i\in S} f_i(x), \qquad \nabla_{S} F(x)=\frac{1}{N_S}\sum_{{i\in S}} \nabla f_i(x).
$$

Thus, Problem (\ref{Problem}) can be tackled using the stochastic gradient method (SGD). The iterates of such a method with a batch size $N_{S^k}$ are given by:
\begin{equation}
x^{(k+1)} = x^{(k)} - \alpha_k \nabla_{S^k} F(x^{(k)}),
\label{sgd}
\end{equation}
where $\alpha_k$ is the learning rate. {\color{black} In what follows we will consider the properties of the batch gradient over a given sequence of iterates $\{x^{(k)}\}$ of (\ref{sgd}), where $x^{(k)}\in \mathbb{R}^d$ for any $k$, and denote by $\text{E}_k$ the conditional expectation with respect to the sigma-algebra generated by $x^{(0)},...x^{(k)}$, and $\text{E}$} the expectation with respect to the discrete measure induced only by the random choice of the batch $S^k$.

\par\medskip\noindent
Recent works \cite{franchini2023line,franchini2023learning,franchini2024stochastic} have introduced efficient algorithms within the SGD framework. The key feature of these approaches is that global convergence is guaranteed by satisfying the following variance control condition (see, for instance, Theorem 2.2 in \cite{franchini2023learning}).

{\color{black}
\par\medskip\noindent
{\bf Assumption 1. }{\it  For a given sequence of iterates $\{x^{(k)}\}$ 
\begin{equation}\label{varbound}
    \text{Var}_k[\nabla_{S_k} F(x)] = \text{E}_k[\|\nabla_{S_k} F(x)-\nabla F(x)\|^2] = \text{E}[\|\nabla_{S^k} F(x^{(k)})-\nabla F(x^{(k)})\|^2]\le \varepsilon_k, \quad \forall k
\end{equation}
    where  $\sum_k\varepsilon_k<\infty$.
 }
\par\medskip\noindent
}


This assumption requires the batch size be updated dynamically during the algorithm's iterations, depending on the variance statistics of the batch selection. The details of how the batch size is updated are discussed in the following sections that address the sampling scheme with replacement and with no replacement.  


\subsection{Sampling With Replacement}

In the sampling with replacement scheme, any batch $S_k^{(c)}$ is formed by selecting an item from the population of $N$ elementary functions $f_i(x)$ and reinserting it before choosing the next one, until $N_{S}$ items are collected. In this way, each item is chosen independently of the others. The gradient estimate for this batch is given by:
\begin{equation}
\nabla_{S^c} F(x) = \frac{1}{N_S} \sum_{i \in S^c} \nabla f_i(x).
\end{equation}

The statistical properties of this estimate can be analyzed using the sampling distribution. The number of different batches with replacement is $N_B^c = \binom{N + N_S - 1}{N_S}$ (i.e., the number of combinations with repetitions). The sampling distribution over the set of batches is uniform, with each batch having a probability of $1/N_B^c$. The following proposition holds.
{\color{black}
\begin{proposition}\label{prop1}
    For a given sequence of iterates $\{x^{(k)}\}$, the gradient batch estimate $\nabla_{S^{(c)}}F(x)$ has the following properties:
\begin{align}
\text{E}_k[\nabla_{S_k^{(c)}} F(x)] &= \text{E}[\nabla_{S_k^{(c)}} F(x^{(k)})]=\sum_{j=1}^{N_B^c}\frac{1}{N_B^c}\left(\frac{1}{N_S} \sum_{i\in S_{k,j}^{(c)}} \nabla f_i(x^{(k)})\right) = \nabla F(x^{(k)}),\label{biasc}\\ 
\text{Var}_k[\nabla_{S_k^{(c)}} F(x)] &=\text{E}[\|\nabla_{S_{k}^{(c)}} F(x^{(k)}) - \nabla F(x^{(k)})\|^2] =\sum_{j=1}^{N_B^c}\frac{1}{N_B^c}\|\nabla_{S_{k,j}^{(c)}} F(x^{(k)}) - \nabla F(x^{(k)})\|^2 = \frac{\text{Var}[\nabla f_\imath(x^{(k)})]}{N_S},\label{varc}
\end{align}
where $\text{Var}[\nabla f_\imath(x^{(k)})] = \frac{1}{N}\sum_{i=1}^N\|\nabla f_i(x^{(k)}) - \nabla F(x^{(k)})\|^2$ is the variance of the individual component gradients computed at $x^{(k)}$.
\end{proposition}

\begin{proof}
The proof is standard and uses arguments reported in the proof of Proposition \ref{prop2}, in a more general setting. Just note that the random draw of $S_k^n$ happens \textbf{after} $x^{(k)}$ is fixed, therefore conditioning on 
$\mathcal{F}_k$ fixes the functions argument to $x^{(k)}$ so that the expectation is taken according to the discrete measure induced by the random selection of the batch.
\end{proof}
}
From equations (\ref{biasc}), (\ref{varc}) it follows that $\text{E}[\nabla_{S^{(c)}} F(x)]$ is an unbiased estimate of $\nabla F(x)$, with variance that decreases as the batch size $N_S$ increases. However, applying the variance bound condition (\ref{varbound}) to equation (\ref{varc}), we obtain:
\begin{equation}
\label{deeplisa}
    N_{S_k} \geq \frac{\text{Var}[\nabla f_\imath(x)]}{\varepsilon_k}.
\end{equation}

This result implies that $N_{S_k}$ must increase without bound as 
$\varepsilon_k \to 0$, which is impractical and unrealistic, as batch sizes cannot be unbounded in real-world scenarios. This issue can be mitigated considering an alternative sampling scheme without replacement, which is discussed in the next section.

\subsection{Sampling Without Replacement}

In the sampling without replacement scheme, each batch $S^n$ is selected by collecting $N_S$ elements simultaneously. Unlike the previous scheme, the $N_S$ selections are no longer independent, and the number of different batches is given by $N_B^n={N\choose N_S}$ (i.e., the number of combinations without repetition). The sampling distribution over the set of batches is again uniform, each batch having a probability of $1/N_B^n$.

\begin{proposition} \label{prop2}
For a given sequence of iterates $\{x^{(k)}\}$, the gradient batch estimate 
$\nabla_{S^n}F(x)$ has the following properties:
\begin{align}
    \text{E}_k[\nabla_{S_k^{(n)}} F(x)] &=\text{E}[\nabla_{S_k^{(n)}} F(x^{(k)})] =\nabla F(x),\label{biasn}\\ 
    \text{Var}_k[\nabla_{S_k^{(n)}} F(x)] &= \text{E}[\|\nabla_{S_k^{(n)}} F(x^{(k)}) - \nabla F(x^{(k)})\|^2]=\frac{\text{Var}[\nabla f_\imath(x)]}{N_S} \frac{N - N_S}{N-1}.
    \label{varn}
\end{align}
\end{proposition}



\begin{proof} 


According to the sampling distribution of the no-replacement scheme we can write:
$$
\text{E}[\nabla_{S_k^{(n)}} F(x^{(k)})]=\sum_{j=1}^{N_B^n}\frac{1}{N_B^n}\nabla_{S_{k,j}^{(n)}} f(x^{(k)}),$$
so that:
$$\text{E}[\nabla_{S_k^{(n)}} F(x^{(k)})]= \sum_{j=1}^{N_B^n}\frac{1}{N_B^n}\frac{1}{N_S} 
\sum_{i\in S_{k,j}^{(n)}} \nabla f_i(x^{(k)})=\frac{1}{N_B^n}\frac{1}{N_S}\sum_{j=1}^{N_B^n}\sum_{i\in S_{k,j}^{(n)}} \nabla f_i(x^{(k)}).
$$
Now, let us compute the number of terms in the double summation. Observe that each batch differs from the others by at least one element. Therefore, each element, say $\nabla f_h(x)$, appears in ${N-1}\choose{N_S-1}$ different batches, since we can
choose any combination of the remaining $N_S-1$ elements from the population of $N-1$. Hence, we can write:
\begin{equation}
    \sum_{j=1}^{N_B^n}\sum_{i\in S_{k,j}^n} \nabla f_i(x) = {{N-1}\choose{N_S-1}} \sum_{i=1}^N \nabla f_i(x).\label{sum}
\end{equation}

Thus, we have:
\begin{align*}
    \text{E}[\nabla_{S^n} F(x)]&=\frac{(N-N_S)!N_S!}{N!}\frac{N}{N_S}\frac{(N-1)!}{(N-N_S)!(N_S-1)!}\frac{1}{N}\sum_{i=1}^N \nabla f_i(x) = \frac{1}{N}\sum_{i=1}^N \nabla f_i(x) =\nabla F(X),
\end{align*}
and (\ref{biasn}) is proved. 
Now, for the variance, we can write:
\begin{align*}
    &  \text{E}[\|\nabla_{S_k^{(n)}} F(x^{(k)})-\nabla_ F(x^{(k)})\|^2] = \text{E}\Bigg[\Big\Vert \frac{1}{N_S}\sum_{i\in S_k^{(n)}} \nabla f_i(x^{(k)}) - \nabla F(x^{(k)}) \Big\Vert^2\Bigg] \\
    =& \text{E}\Bigg[ \frac{1}{N_S^2}\sum_{i\in S_k^{(n)}} \|\nabla f_i(x^{(k)}) - \nabla F(x^{(k)})\|^2 \Bigg]+\text{E}\Bigg[\frac{1}{N_S^2}\sum_{i\neq h\in S_k^{(n)}} (\nabla f_i(x^{(k)}) - \nabla F(x^{(k)}))^T(\nabla f_h(x^{(k)}) - \nabla F(x^{(k)})) \Bigg].
\end{align*}
For the first term, we can write\footnote{In the case of sampling with replacement, the second term is zero because of the independence of the batch items. Therefore result in (\ref{varc}) is straightforward.}:
\begin{align*}
    \text{E}&\Bigg[\frac{1}{N_S^2}\sum_{i\in S_k^{(n)}} \|\nabla f_i(x^{(k)}) - \nabla F(x^{(k)})\|^2 \Bigg]=\frac{1}{N_B^n}\sum_{j=1}^{N_B^n}\frac{1}{N_S^2}\sum_{i\in S_{k,j}^{(n)}} \| \nabla f_i(x^{(k)})-\nabla F(x^{(k)})\|^2\\
    =& \frac{1}{N_S^2}\frac{1}{N_B^n}{{N-1}\choose{N_S-1}}\sum_{i=1}^N\|\nabla f_i(x^{(k)}) - \nabla F(x^{(k)})\|^2=\frac{1}{N_S}\frac{1}{N}\sum_{i=1}^N \|\nabla f_i(x^{(k)}) - \nabla F(x^{(k)})\|^2=\frac{\text{Var}[\nabla f_{\imath}(x^{(k)})]}{N_S}.
\end{align*} 
where we used (\ref{sum}) in the last step. For the second term, we have:
\begin{align*}
&\text{E}\Bigg[\frac{1}{N_S^2}\sum_{i\neq h\in S_k^{(n)}} (\nabla f_i(x^{(k)}) - \nabla F(x^{(k)})^T(\nabla f_h(x^{(k)}) - \nabla F(x^{(k)}) \Bigg]\\
&=\frac{1}{N_B^n}\frac{1}{N_S^2}\sum_{j=1}^{N_B^n}\sum_{i\neq h\in S_{k,j}^{(n)}} (\nabla f_i(x^{(k)}) - \nabla F(x^{(k)}))^T(\nabla f_h(x^{(k)}) - \nabla F(x^{(k)})).
\end{align*}
For any $j$, the inner sum has $2 {N_s\choose 2}$ terms, determining the covariance
 $\text{Cov}_{S_{k,j}^{(n)}}[\nabla f_{i_k}(x^{(k)}),\nabla f_{h_k}(x^{(k)})]$ of the items within batch $S_{k,j}^n$:
$$
\text{Cov}_{S_{k,j}^{(n)}}[\nabla f_{i_k}(x^{(k)}),\nabla f_{h_k}(x^{(k)})]=\frac{1}{2{N_s\choose 2}}\sum_{{i}\neq {h}\in S_{k,j}^{(n)}} (\nabla f_{i}(x^{(k)}) - \nabla F(x^{(k)}))^T(\nabla f_{h}(x^{(k)}) - \nabla F(x^{(k)})),
$$
so that:
$$
\text{Cov}_{S_k^{(n)}}[\nabla f_\imath(x^{(k)}),\nabla f_\jmath(x^{(k)})] =\frac{1}{N_B^n}\sum_{j=1}^{N_B^n} \text{Cov}_{S_{k,j}^n}[\nabla f_{i_k}(x^{(k)}),\nabla f_{h_k}(x^{(k)})]
$$
represents the \textit{average batch covariance} of the component gradients. Hence, we can write
\begin{equation}
\text{Var}[\nabla_{S_k^{(n)}} F(x^{(k)})]=
    \frac{Var[\nabla f_{\imath}(x^{(k)})]}{N_S} + \frac{1}{N_s^2}{2{N_s\choose 2}}\,\text{Cov}_{S_k^{(n)}}[\nabla f_\imath(x^{(k)}),\nabla f_\jmath(x^{(k)})]. 
    \label{bcovar}
\end{equation}
Now, if $N_S=N$, it follows that  $\text{Var}[\nabla_{S_k^{(n)}} F(x^{(k)})]=0$, so that:
\begin{align}
    0 &= \frac{\text{Var}[\nabla f_{\imath}(x^{(k)})]}{N} + \frac{1}{N^2}{2{N\choose 2}}\,\text{Cov}_{S_k^{(n)}}[\nabla f_\imath(x^{(k)}),\nabla f_\jmath(x^{(k)})],\quad\textit{from which}\\ &\text{Cov}_{S_k^{(n)}}[\nabla f_\imath(x^{(k)}),\nabla f_\jmath(x^{(k)})] = -\frac{\text{Var}[\nabla f_{\imath}(x^{(k)})]}{N-1}.
\end{align}

By taking into account (\ref{bcovar}), we finally obtain:
\begin{align*}
    \text{Var}_k[\nabla_{S_k^{(n)}} F(x)] = \text{Var}[\nabla_{S_k^{(n)}} F(x^{(k)})] =& \frac{\text{Var}[\nabla f_{\imath}(x^{(k)})]}{N_S} - \frac{N_S-1}{N_S}\frac{\text{Var}[\nabla f_{\imath}(x^{(k)})]}{N-1}=\frac{\text{Var}[\nabla f_{\imath}(x^{(k)})]}{N_S}\frac{N-N_S}{N-1}.
\end{align*}

Thus, the variance result is proved. 
\end{proof} 

The estimate $\nabla_{S^{(n)}} F(x)$ is unbiased, and its variance decreases as the batch size $N_S$ increases, and tends to zero as $N_S$ approaches the total population size $N$. \\ Now applying the variance bound condition from (\ref{varbound}) to the variance in (\ref{varn}), we derive the following batch size condition:

\begin{equation}
    N_{S_k} \ge N \frac{Var[\nabla f_{\imath}(x)]}{\displaystyle{(N-1)\varepsilon_k+{Var[\nabla f_{\imath}(x)]}}}.
 \label{ABSS}
\end{equation}

Comparing (\ref{ABSS}) with (\ref{deeplisa}) we can appreciate that the r.h.s does not increase without bound as $\varepsilon_k\to 0$, and is always less than $N$.  Moreover, by construction $N_{S_k}\le N$. The update of $N_{S_k}$ is smoother compared to that of the replacement-based sampling scheme, as it  is clearly illustrated in Figure~\ref{fig:batch_size_growth}. In practice, whether using formula (\ref{deeplisa}) as in \cite{franchini2024stochastic} or formula (\ref{ABSS}), determining the batch size requires imposing an upper bound $C$ on the value of $\text{Var}[\nabla f_{\imath}(x)]$. In this example, we set $C = 10$ and \( N = 30,000 \), and it is important to note that changing \( N \) and \( C \) does not affect the overall shape of the curves. For the batch size based on formula (\ref{deeplisa}), we plot $\textcolor{red}{\min}(N, N_{S_k})$ to prevent it from becoming excessively large and difficult to compare with the other method. The graph clearly shows that the batch size determined by formula (\ref{ABSS}) grows more gradually compared to the one based on formula (\ref{deeplisa}). This slower and more controlled growth may result in more stable training dynamics, especially in situations where rapid increases in batch size can negatively affect convergence or computational efficiency.


\begin{figure}[ht]
    \centering
    \includegraphics[width=0.6\textwidth]{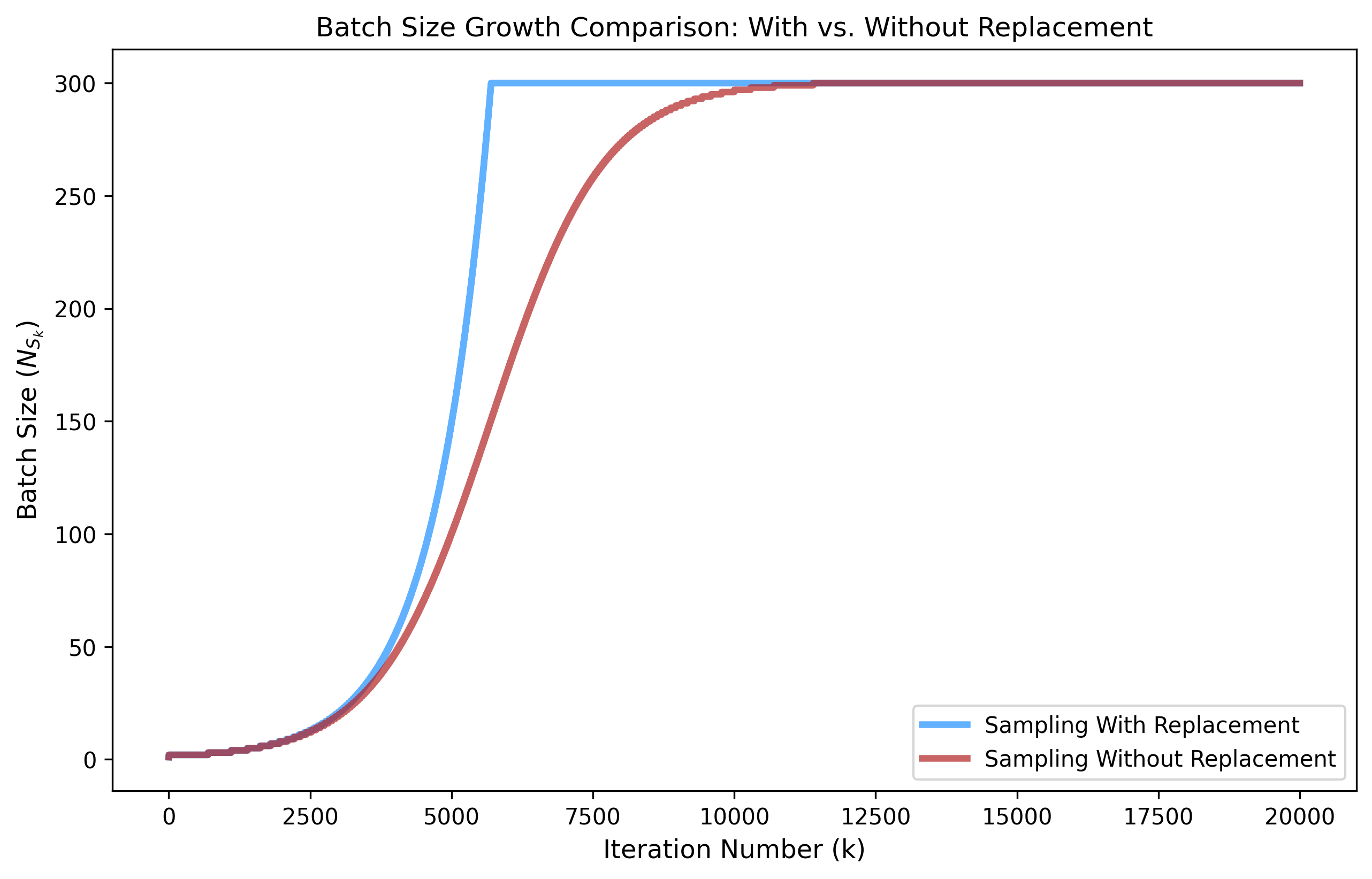}
     \caption{{\color{black}Batch Size Growth Comparison. Red: Sampling Without Replacement. Blue: Sampling With Replacement (truncated at N)}}
    \label{fig:batch_size_growth}
\end{figure}

\section{Conclusions}
In this paper, we introduced a novel approach for batch size selection in stochastic gradient methods based on sampling without replacement. This method addresses the limitations of existing approaches that rely on sampling with replacement, which can be unrealistic and may lead to unbounded batch sizes. Our approach ensures that the batch size remains bounded by the population size and has a smoother update, while maintaining a sound theoretical framework.

Future work could explore applying our method to practical optimization tasks and conducting empirical evaluations to assess its performance in real-world scenarios. Additionally, extending the theoretical analysis to more complex stochastic optimization algorithms or different sampling schemes could further enhance its applicability.

\section{Acknowledgements}
This research did not receive any specific grant from funding agencies in the public, commercial, or not-for-profit sectors.




\bibliographystyle{cas-model2-names}

\bibliography{mybib}



\end{document}